\documentclass{amsart}
\usepackage{amssymb}
\usepackage{amsfonts}
\usepackage[all]{xy}
\usepackage{amssymb}
\usepackage{amsmath}
\usepackage{amsthm}
\usepackage{enumerate}
\usepackage{tabularx}
\usepackage{pdfpages}
\usepackage{centernot}
\usepackage{mathtools}
\usepackage{stmaryrd}
\usepackage{amsthm,amssymb}
\usepackage{etoolbox}
\usepackage{url}
\usepackage{tikz}
\usepackage{amssymb}
\usetikzlibrary{matrix}
\usepackage{tikz-cd}
\usepackage{tikz}
\usepackage{marginnote}
\definecolor{mygray}{gray}{0.85}
\usepackage[backgroundcolor=mygray,colorinlistoftodos,prependcaption,textsize=small]{todonotes}
\usepackage{xargs}                      
\let\hat\widehat

\newcommand{\mrm}[1]{\mathrm{#1}}

\renewcommand{\leq}{\leqslant}
\renewcommand{\geq}{\geqslant}

\makeatletter
\def\subsection{\@startsection{subsection}{3}%
  \z@{.5\linespacing\@plus.7\linespacing}{.3\linespacing}%
  {\bfseries\centering}}
\makeatother

\makeatletter
\def\subsubsection{\@startsection{subsubsection}{3}%
  \z@{.5\linespacing\@plus.7\linespacing}{.3\linespacing}%
  {\centering}}
\makeatother

\makeatletter
\def\myfnt{\ifx\protect\@typeset@protect\expandafter\footnote\else\expandafter\@gobble\fi}
\makeatother

\newtheorem{theorem}{Theorem}[section]

\theoremstyle{definition} 
\newtheorem{convention}[theorem]{Convention}
\newtheorem{corollary}[theorem]{Corollary}
\newtheorem{lemma}[theorem]{Lemma}
\newtheorem{proposition}[theorem]{Proposition}

\newtheorem{context}[theorem]{Context}

\theoremstyle{definition}

\newtheorem{fact}[theorem]{Fact}
\newtheorem{definition}[theorem]{Definition}
\newtheorem{remark}[theorem]{Remark}

\newtheorem{notation}[theorem]{Notation}

\usepackage[backgroundcolor=mygray,colorinlistoftodos,prependcaption,textsize=small]{todonotes}
\usepackage{xcolor}

\newcounter{claimcounter}
\numberwithin{claimcounter}{theorem}
\newenvironment{claim}{\refstepcounter{claimcounter}{\noindent {\underline{\em Claim \theclaimcounter}.}}}{}
\newenvironment{claimproof}[1]{\noindent{{\em Proof.}}\space#1}{\hfill $\rule{0.40em}{0.40em}$}

\newcommand{\pureindep}[1][]{%
  \mathrel{
    \mathop{
      \vcenter{
        \hbox{\oalign{\noalign{\kern-.3ex}\hfil$\vert$\hfil\cr
              \noalign{\kern-.7ex}
              $\smile$\cr\noalign{\kern-.3ex}}}
      }
    }\displaylimits_{#1}
  }
}

\begin{document}

\begin{abstract} 
We prove that every non-Archimedean Polish group is topologically
isomorphic to the outer automorphism group of a countable discrete
group. Furthermore, our construction is Borel. This implies that determining whether two countable discrete groups have topologically isomorphic outer automorphism groups is at least as complicated as classifying non-Archimedean Polish groups up to topological isomorphism, and so in particular not classifiable by countable structures.
Our construction is based on the theory of right-angled Coxeter
groups. Along the way, we prove results of independent interest on the topological group $\mrm{Aut}(W)$,
for $W$ a countable right-angled Coxeter group. This has purely group-theoretic applications; in particular, we give a 
graph-theoretic characterization of when the
group $\mrm{Spe}(W)$ of special automorphisms of $W$ coincides with $\mrm{Inn}(W)$. Combined with previous work, this gives a characterization of when $\mrm{Aut}(W)$ is inner-by-graph, for $W$ countable.
\end{abstract}

\title[Realizing non-Archimedean Polish groups]{Realizing non-Archimedean Polish groups as outer automorphism groups}

\thanks{Jean-Luc Rabideau would like to thank Denis Osin for introducing this problem to him and for sharing his ideas regarding a small-cancellation proof of the main theorem. Research of Gianluca Paolini was  supported by project PRIN 2022 ``Models, sets and classifications{''}, prot. 2022TECZJA, and by INdAM Project 2024 (Consolidator grant) ``Groups, Crystals and Classifications''.}

\author{Gianluca Paolini}
\address{Department of Mathematics ``Giuseppe Peano'', University of Torino, Via Carlo Alberto 10, 10123, Torino, Italy.}
\email{gianluca.paolini@unito.it}

\author{Jean-Luc Rabideau}
\address{Department of Mathematics, Vanderbilt University, 1326 Stevenson Center Ln, Nashville, TN 37240}
\email{jean-luc.p.rabideau@vanderbilt.edu}


\date{\today}
\maketitle
\tableofcontents

\section{Introduction}
A Polish group $P$ is called \textit{non-Archimedean} if it admits a countable neighborhood basis of the identity consisting of open subgroups. Equivalently, $P$ is topologically isomorphic to a closed subgroup of $\mrm{Sym}(\omega)$, the permutation group of $\omega$ in the pointwise convergence topology. Non-Archimedean Polish groups are exactly the automorphism groups $\mrm{Aut}(M)$ of countable structures $M$ in the pointwise convergence topology (see, e.g., \cite[1.5]{Becker_Kechris_1996} for both equivalences). It is then natural to ask whether every such group can be realized as $\mrm{Aut}(M)$ for $M$ chosen from a more restricted class of countable structures, such as the countable discrete groups.

It is not hard to show, however, that cyclic groups of odd order are never isomorphic to $\mrm{Aut}(C)$ for any discrete group $C$. Instead, for a countable discrete group $C$, we consider the outer automorphism groups \mbox{$\mrm{Out}(C) = \mrm{Aut}(C)/\mrm{Inn}(C)$}. In the quotient topology, these are non-Archimedean Polish groups --- equivalently, Hausdorff --- exactly when $\mrm{Inn}(C)$ is a closed subgroup of $\mrm{Aut}(C)$. Our main result is that all non-Archimedean Polish groups arise in this way. 

\begin{theorem}\label{thm:main}
    A Hausdorff topological group $G$ realizes as $\mrm{Out}(C) = \mrm{Aut}(C) / \mrm{Inn}(C)$ in the quotient topology for a countable group $C$ if and only if $G$ is non-Archimedean Polish.
\end{theorem}

Theorem \ref{thm:main} follows from the stronger Theorem \ref{thm:stronger} below. Let $\cong_{\mrm{top}}$ denote isomorphism of topological groups. For every closed subgroup $G \leq \mrm{Sym}(\omega)$, we construct a countably infinite group $f(G)$ such that $\mrm{Out}(f(G)) \cong_{\mrm{top}} G$. In particular, Theorem \ref{thm:stronger} states that the map $G \mapsto f(G)$ from the Effros space of closed subgroups of $\mrm{Sym}(\omega)$ to the space of groups with domain $\omega$ is Borel.

\smallskip
As a corollary, using the main result of \cite{gao-nies-paolini}, we obtain the following.

\begin{corollary} Deciding whether two countable groups have topologically isomorphic outer automorphism groups is as hard as classifying non-Archimedean Polish groups up to topological isomorphism; in particular, the equivalence relation $\ell_\infty$ of bounded difference between sequences in $\mathbb{R}^\omega$ is Borel reducible to this problem.
\end{corollary}

It was first shown by Matumoto \cite{Matumoto1989} that every group $H$ arises as $\mrm{Out}(K)$ for some group $K$, in the purely algebraic sense. This result has since been improved by restricting the class of groups $H$. For example, every group $H$ realizes as $\mrm{Out}(K)$, where $K$ is simple \cite{DROSTE_GIRAUDET_GOBEL_2001}, metabelian \cite{GoebelParas2000}, or finitely generated for $H$ countable~\cite{bumagin_wise}. 

However, there are two main complications in trying to use previously known methods to prove Theorem \ref{thm:main}. The first, which has already been mentioned, concerns the topology of $\mrm{Out}(C)$. In general, $\mrm{Inn}(C)$ is not necessarily closed in $\mrm{Aut}(C)$, so the quotient topology on $\mrm{Out}(C)$ need not be Hausdorff. For example, if $C$ is the group of finitely supported permutations of $\omega$, then $\mrm{Inn}(C) = C$, but its closure in $\mrm{Aut}(C)$ is $\mrm{Aut}(C) = \mrm{Sym}(\omega)$ (see \cite[Theorem 8.2A]{DixonMortimer1996}). Hence, the quotient topology on $\mrm{Out}(C)$ is indiscrete, and in particular not Hausdorff.

The second issue arises from the restriction on the cardinality of $C$, even if we disregard topology. In general, a countable group $C$ can have at most $2^{\aleph_0}$ automorphisms. However, the cardinality of $\mrm{Out}(C)$ is either $\leq \aleph_0$ or $2^{\aleph_0}$. Indeed, in this case $\mrm{Aut}(C)$ is Polish and the coset equivalence relation $\alpha \sim \beta \iff \alpha\beta^{-1} \in \mrm{Inn}(C)$ is Borel. Hence, Silver's Dichotomy \cite{SILVER19801} implies $\mrm{Aut}(C)/\mrm{Inn}(C)$ is either countable or of cardinality continuum. Therefore, additional assumptions on $G$ (beyond the obvious $|G| \leq 2^{\aleph_0}$) are necessary, even if one only aims to realize it as $\mrm{Out}(C)$ for a countable group $C$ in a purely algebraic (non-topological) sense.

Our construction is based on the theory of right-angled Coxeter groups, and the well-known observation that all non-Archimedean Polish groups are automorphism groups of graphs (see Fact \ref{second_fact}). Namely, if $W(\Gamma)$ is a right-angled Coxeter group on a graph $\Gamma$ (see Section \ref{sec_preliminaries}), a fundamental result of Tits \cite{tits} shows that $\mrm{Aut}(W(\Gamma))$ decomposes as a semidirect product of two subgroups $\mrm{Spe}(W(\Gamma))$ and $F(\Gamma)$ (see Definitions \ref{spe_def} and \ref{def_FGamma}, respectively). We say a decomposition $G = K \rtimes H$ of a topological group $G$ is a \textit{topological decomposition} if the topology on $G$ agrees with the product topology on $K \times H$. Towards proving Theorem \ref{thm:main}, we also prove:

\begin{theorem}\label{theorem_tits_topological}
    Let $W=W(\Gamma)$ be a countable right-angled Coxeter group. Then the Tits decomposition $\mrm{Aut}(W(\Gamma)) = \mrm{Spe}(W(\Gamma)) \rtimes F(\Gamma)$ is a topological decomposition. In particular, both $\mrm{Spe}(W(\Gamma))$ and $F(\Gamma)$ are closed subgroups of $\mrm{Aut}(W(\Gamma))$.
\end{theorem}

One of the contributions of this paper concerns the equality
$\mrm{Spe}(W(\Gamma)) = \mrm{Inn}(W(\Gamma))$. Together with the equality
$F(\Gamma) = \mrm{Aut}(\Gamma)$, it governs, via Tits' result, whether
$\mrm{Aut}(W(\Gamma))$ is inner-by-graph --- a long-established line of
research in Coxeter group theory; see e.g.\ \cite{franzsen-howlett-muhlherr}
and references therein. The equality
$F(\Gamma) = \mrm{Aut}(\Gamma)$ is completely understood (see
Fact~\ref{fact_F_gamma_star_property}), so we focus on the other one. When
$\Gamma$ is finite, $\mrm{Spe}(W(\Gamma)) = \mrm{Inn}(W(\Gamma))$ holds if and
only if $\Gamma$ is star-connected (see Definition~\ref{def_starconnected} and
Fact~\ref{spe_inn}). For infinite $\Gamma$, star-connectedness remains necessary
but is no longer sufficient, as Hyttinen and the first author observed in \cite[Theorem~2.18]{pao2}.

\smallskip
Our approach is to separate the equality into its two topological ingredients,
the density and the closedness of $\mrm{Inn}(W(\Gamma))$ in
$\mrm{Spe}(W(\Gamma))$, and to characterize each for countable $\Gamma$. On the
density side, Theorem~\ref{density_theorem} shows that star-connectedness alone
forces $\mrm{Spe}(W(\Gamma)) = \overline{\mrm{Inn}(W(\Gamma))}$. On the
closedness side, we introduce the finite-star-base property
(Definition~\ref{definition_finite_star_base}) and prove in
Theorem~\ref{theorem:inn_closed} that, for countable $\Gamma$, it holds exactly
when $\mrm{Inn}(W(\Gamma))$ is closed. Since $\mrm{Spe}(W(\Gamma)) =
\mrm{Inn}(W(\Gamma))$ amounts to $\mrm{Inn}(W(\Gamma))$ being simultaneously
dense and closed in $\mrm{Spe}(W(\Gamma))$, the two theorems combine to give the
following complete characterization.

\begin{theorem}\label{theorem_spe_inn} Let $W = W(\Gamma)$ be a countable right-angled Coxeter group. Then the following are equivalent. \begin{enumerate}[(1)]
    \item\label{theorem_spe_inn_1} $\mrm{Spe}(W) = \mrm{Inn}(W)$.
    \item\label{theorem_spe_inn_2} $\Gamma$ is star-connected and $\mrm{Inn}(W)$ is closed in $\mrm{Aut}(W)$.
    \item\label{theorem_spe_inn_3} $\Gamma$ is star-connected and $\Gamma$ has the finite-star-base property (cf. \ref{definition_finite_star_base}).
\end{enumerate} 
\end{theorem}

    \begin{corollary} Let $W = W(\Gamma)$ be a countable right-angled Coxeter group. Then $\mrm{Aut}(W)$ is inner-by-graph, i.e., $\mrm{Aut}(W) = \mrm{Inn}(W) \rtimes \mrm{Aut}(\Gamma)$, if and only if $\Gamma$ has the star-property (cf.\ \ref{def_starconnected}), is star-connected, and has the finite-star-base property.
    \end{corollary}

    Finally, we believe that it is worth noting that, for star-connected $\Gamma$, the finite-star-base
property is equivalent to $\mrm{Spe}(W(\Gamma))$ being countable. Indeed, if
$\mrm{Spe}(W(\Gamma))$ is countable, then, since $\mrm{Spe}(W(\Gamma))$ is a Polish group, it is discrete
by the Baire category theorem, so $\mrm{Inn}(W(\Gamma))$ is closed and $\Gamma$
has the finite-star-base property by Theorem~\ref{theorem:inn_closed}.
Conversely, if $\Gamma$ has the finite-star-base property, then
$\mrm{Inn}(W(\Gamma))$ is closed, hence equal to $\mrm{Spe}(W(\Gamma))$ by
Theorem~\ref{theorem_spe_inn}, and so $\mrm{Spe}(W(\Gamma))$ is countable.

\medskip

Concerning the structure of the paper, in Section \ref{sec_preliminaries}, we recall some basic facts and results about right-angled Coxeter groups. In Section \ref{section_tits_topological}, we prove that Tits' decomposition is topological (Theorem \ref{theorem_tits_topological}). In Section \ref{Section_inn_closed}, we introduce the finite-star-base property and prove Theorem \ref{theorem:inn_closed}. In Section \ref{sec_Spe=Inn}, we characterize when $\mrm{Inn}(W) = \mrm{Spe}(W)$, proving Theorems \ref{density_theorem} and \ref{theorem_spe_inn}. Section \ref{section_main_theorem} proves Theorem \ref{thm:main}.


\section{Preliminaries}\label{sec_preliminaries}

\begin{notation}\label{N(v)_notation} A graph $\Gamma = (V(\Gamma), R_\Gamma)$ is a set with a binary irreflexive and symmetric relation. Given $v \in V(\Gamma)$, we denote by  $N_{\Gamma}(v) = N(v)$ the set $\{u \in V(\Gamma) : v R_\Gamma u\}$ and by $N^*_{\Gamma}(v) = N^*(v)$ the set $\{v\} \cup N_{\Gamma}(v)$.
\end{notation}

Given a graph $\Gamma = (V(\Gamma), R_\Gamma)$, we denote by $W(\Gamma)$ the associated right-angled Coxeter group, that is, the group given by the following presentation: the generating set of $W(\Gamma)$ is $V(\Gamma)$, each element in $V(\Gamma)$ has order $2$, and for each edge $\{v, u\}$ of the graph $\Gamma$ we impose the relation that $u$ and $v$ commute. 

A fundamental result of Tits \cite{tits} gives an explicit description of $\mrm{Aut}(W(\Gamma))$ as a semidirect product of two subgroups of $\mrm{Aut}(W(\Gamma))$, namely $\mrm{Spe}(W(\Gamma))$ and $F(\Gamma)$.

	\begin{definition}\label{spe_def} Let $W$ be a right-angled Coxeter group. We denote by $\mrm{Spe}(W)$ the set of automorphisms $\alpha \in \mrm{Aut}(W)$ such that for every involution $h \in W$ there exists $g \in W$ such that $\alpha(h) = ghg^{-1}$.
\end{definition}

	\begin{definition}\label{def_FGamma} Let $\Gamma$ be a graph. We think of the set of finite subsets of $\Gamma$ as a $GF(2)$-vector space (the field with $2$ elements) $T(\Gamma) = (\mathcal{P}_{fin}(\Gamma), \triangle, \cdot)$ by letting:
\begin{enumerate}[(1)]
\item $S_1 \triangle S_2 = (S_1 - S_2) \cup (S_2 - S_1)$;
\item $0 \cdot S = \emptyset$;
\item $1 \cdot S = S$.
\end{enumerate}
We denote by $F(\Gamma)$ the set of linear automorphisms of $T(\Gamma)$ which send finite cliques of $\Gamma$ to finite cliques of $\Gamma$.
\end{definition}

\begin{fact}[Tits \cite{tits}]\label{fact_tits}
Let $\Gamma$ be a graph. Then: $$\mrm{Aut}(W(\Gamma)) = \mrm{Spe}(W(\Gamma)) \rtimes F(\Gamma).$$
\end{fact}

\begin{definition}\label{partial_conj} Let $\Gamma$ be a graph, $s \in \Gamma$ and $C$ a union of connected components of $\Gamma \setminus N^*(s)$. We define an automorphism (cf. Fact \ref{partial_conj_fact}) $\pi_{(s, C)}$ of $W(\Gamma)$ as follows:
$$\begin{cases} \pi_{(s, C)}(t) = sts \;\;\;\; \text{ if } t \in C \\
			  \pi_{(s, C)}(t) = t \;\;\;\;\;\;\; \text{ otherwise. }
\end{cases} $$
Automorphisms of the form $\pi_{(s, C)}$ are called {\em partial conjugations}.
\end{definition}

	\begin{fact}[\cite{muhlherr}]\label{partial_conj_fact} Let $\Gamma$ be a graph. Then the partial conjugations (cf. Def. \ref{partial_conj}) are automorphisms of $W(\Gamma)$ and, if $\Gamma$ is {\em finite}, then $\mrm{Spe}(W(\Gamma))$ is generated by them.
\end{fact}
 
	\begin{definition}\label{def_starconnected} We say that a graph $\Gamma$ is star-connected if for every $v \in V(\Gamma)$, the induced subgraph $\Gamma \setminus N^*(v)$ is connected. We say that $\Gamma$ has the star-property if for every $v \neq u \in V(\Gamma)$ we have that $N^*(v) \not\subseteq N^*(u)$.
\end{definition}

	\begin{fact}[{\cite[Comm. 3]{castella}}]\label{fact_F_gamma_star_property} Let $\Gamma$ be a graph. Then the following are equivalent:
	\begin{enumerate}[(1)]
	\item $F(\Gamma) = \mrm{Aut}(\Gamma)$;
	\item $\Gamma$ has the star-property.
\end{enumerate}
\end{fact}

\begin{fact}[{\cite[Comm. 3]{castella}}]\label{spe_inn} Let $\Gamma$ be a {\em finite} graph. The following are equivalent:
	\begin{enumerate}[(1)]
	\item $\mrm{Spe}(W(\Gamma)) = \mrm{Inn}(W(\Gamma))$;
	\item $\Gamma$ is star-connected.
\end{enumerate}
\end{fact}

	As was remarked in \cite[Theorem~2.8(c)]{pao2}, in the general case (i.e., allowing infinite rank Coxeter groups) the star-connectedness of $\Gamma$ is a necessary but not sufficient condition for $\mrm{Spe}(W(\Gamma)) = \mrm{Inn}(W(\Gamma))$. We will look more closely at the problem of giving a graph-theoretic characterization of $\mrm{Spe}(W(\Gamma)) = \mrm{Inn}(W(\Gamma))$ in Section~\ref{sec_Spe=Inn}.

\begin{remark}\label{F_embedding_remark} The group $F(\Gamma)$ is naturally identified with a subgroup of $\mrm{Aut}(W(\Gamma))$: to $f \in F(\Gamma)$ we associate the automorphism $\hat{f} \in \mrm{Aut}(W(\Gamma))$ determined on the standard generators by $
\hat{f}(v) = \prod_{u \in f(\{v\})} u \qquad (v \in V(\Gamma))$,
the product being taken over the finite clique $f(\{v\})$. We will identify $f \in F(\Gamma)$ with $\hat{f} \in \mrm{Aut}(W(\Gamma))$.
\end{remark}

\section{The Tits decomposition of $\mrm{Aut}(W)$ is topological}\label{section_tits_topological}

	Throughout this section, $W = W(\Gamma)$ is a countable right-angled Coxeter group. 

    \begin{definition}
        Let $G$ be a topological group with a decomposition as a semidirect product $G = K \rtimes H$. We say this decomposition is \textit{topological} if the topology on $G$ agrees with the product topology on $K \times H$, where both $K$ and $H$ have their induced subspace topologies.
    \end{definition}
    \begin{remark}
        If a Hausdorff topological group $G$ has a topological decomposition $K \rtimes H$, then both $K$ and $H$ are closed subgroups of $G$.
    \end{remark}

    The purpose of this section is to prove Theorem \ref{theorem_tits_topological}, i.e., that the Tits decomposition in Fact \ref{fact_tits} is topological, where $\mrm{Aut}(W)$ is given the topology of pointwise convergence.

	\begin{proposition}\label{Spe_closed} $\mrm{Spe}(W)$ is closed in $\mrm{Aut}(W)$.
\end{proposition}

	\begin{proof} The natural map $\pi: \mrm{Aut}(W) \rightarrow F(\Gamma)$ is a continuous homomorphism whose kernel is exactly $\mrm{Spe}(W)$. 
	\end{proof}

	\begin{proposition}\label{Fgamma_closed} $F(\Gamma)$ is closed in $\mrm{Aut}(W)$.
\end{proposition}

	\begin{proof} Throughout, $F(\Gamma)$ is regarded as a subgroup of $\mrm{Aut}(W)$ via the identification of Remark~\ref{F_embedding_remark}, so that its elements are the automorphisms $\hat{f}$ with $\hat{f}(v) = \prod_{u \in f(\{v\})} u$ for $f \in F(\Gamma)$. Now, call an element $w \in W$ a {\em clique product} if $w = \prod_{u \in C} u$ for some finite clique $C \subseteq V(\Gamma)$ (the empty product being $e_W$), and write $\mrm{CP}_{V(\Gamma)} = \mrm{CP} \subseteq W$ for the set of all clique products (notice that this depends on the generating set $V(\Gamma)$ but this does not affect our argument). If $C$ is a clique, then $\prod_{u \in C} u$ is a reduced word, so $C = \mrm{sp}(w)$ is recovered as the support of $w$; in particular a clique product is determined by its support, and distinct cliques yield distinct elements. It is easy to see that
$$
F(\Gamma) = \{\alpha \in \mrm{Aut}(W) : \alpha(v) \in \mrm{CP} \text{ for every } v \in V(\Gamma)\} =: X.
$$

\noindent Equivalently, $F(\Gamma)$ is the set of automorphisms of $W$ that send every clique product to a clique product.

\smallskip \noindent It remains to show that $X$ is closed. Recall that $\mrm{Aut}(W)$ carries the topology of pointwise convergence, that is, the subspace topology inherited from the product space $W^W$, where $W$ is given the discrete topology. For each $v \in V(\Gamma)$ consider the evaluation map
\[
\mrm{ev}_v : \mrm{Aut}(W) \to W, \qquad \alpha \mapsto \alpha(v).
\]
This map is continuous: it is the restriction to $\mrm{Aut}(W)$ of the $v$-th coordinate projection $W^W \to W$, which is continuous by the very definition of the product topology. (Concretely, for every $w \in W$ the preimage $\mrm{ev}_v^{-1}(\{w\}) = \{\alpha \in \mrm{Aut}(W) : \alpha(v) = w\}$ is, by definition, a subbasic open set of the topology of pointwise convergence.) Since $W$ is discrete, every subset of $W$ is clopen; in particular $\mrm{CP} \subseteq W$ is clopen, and hence so is its preimage $\mrm{ev}_v^{-1}(\mrm{CP})$. Finally, for $\alpha \in \mrm{Aut}(W)$ we have $\alpha \in \mrm{ev}_v^{-1}(\mrm{CP})$ if and only if $\alpha(v) \in \mrm{CP}$, so unwinding the definition of $X$ gives
\[
X = \{\alpha \in \mrm{Aut}(W) : \alpha(v) \in \mrm{CP} \text{ for every } v \in V(\Gamma)\} = \bigcap_{v \in V(\Gamma)} \mrm{ev}_v^{-1}(\mrm{CP}),
\]
an intersection of closed sets, hence closed. Since $X = F(\Gamma)$, we are done.
	\end{proof}

\begin{lemma}\label{lemma_top_decomp}
    Suppose $G$ is a Polish group, and $H,K$ are closed subgroups such that $G = K \rtimes H$. Then this decomposition is topological.
\end{lemma}
\begin{proof}
    Consider $(K \rtimes H, \mathcal{T}_{prod})$ where $\mathcal{T}_{prod}$ denotes the product topology and both $K$ and $H$ are equipped with their subspace topologies from $G$. Since both $K$ and $H$ are closed subspaces, both are Polish groups. Therefore \mbox{$(K \rtimes H, \mathcal{T}_{prod})$} is a Polish space; it is a Polish group because the group operations \begin{align*}
        (a_1,b_1) (a_2,b_2) &= (a_1 (b_1a_2 b_1^{-1}), b_1 b_2) \\
        (a,b)^{-1} &= (b^{-1} a^{-1} b, b^{-1})
    \end{align*}
    are continuous with respect to $\mathcal{T}_{prod}$. Then, the  multiplication map $(K \rtimes H, \mathcal{T}_{prod}) \to G$ is a continuous bijective group homomorphism between Polish groups, and so is a topological isomorphism by the Open Mapping Theorem for Polish groups (see, e.g., \cite[Theorem 1.2.6]{Becker_Kechris_1996}). Hence, the two topologies on $G$ coincide.
\end{proof}

\begin{proof}[Proof of Theorem \ref{theorem_tits_topological}]
    Since $W$ is countable, $\mrm{Aut}(W)$ is a (non-Archimedean) Polish group. Thus, the theorem follows from Proposition \ref{Spe_closed}, Proposition \ref{Fgamma_closed}, and Lemma \ref{lemma_top_decomp}.
\end{proof}

\section{When is $\mrm{Inn}(W)$ closed in $\mrm{Aut}(W)$?}\label{Section_inn_closed}

The aim of this section is to prove Theorem \ref{theorem:inn_closed}, i.e., a 
graph-theoretic characterization of when $\mrm{Inn}(W)$ is closed in
$\mrm{Aut}(W)$. Our standing assumption in this section is that $W$ is a countable right-angled Coxeter group.

	\begin{notation}\label{Z_notation} Let
	\[
	Z_\Gamma=\{z\in V(\Gamma):N^*(z)=V(\Gamma)\}.
	\]
That is, $Z_\Gamma$ is the set of vertices of $\Gamma$ which are adjacent to every other vertex of $\Gamma$. 
\end{notation}

	\begin{convention}\label{the_convention} The empty intersection of subsets of $V(\Gamma)$ is $V(\Gamma)$.
\end{convention}

\begin{definition}\label{definition_finite_star_base}
    Let $\Gamma$ be a graph. We say $\Gamma$ has the finite-star-base property if there is a finite set $A \subseteq V(\Gamma)$ such that $\bigcap_{a \in A} N^*(a) = Z_{\Gamma}$.
\end{definition}
\begin{remark} Let us look more closely at this graph-theoretic condition. In general, the property states that for every $v \in V(\Gamma)$, either $v$ is adjacent to every vertex, or there is a witness $u \in A$ such that $v$ is not adjacent to $u$.
We observe three cases:
\newline \underline {Case 1}. $A = \emptyset$.
\newline In this case $Z_\Gamma = V(\Gamma)$ and so $W(\Gamma)$ is abelian, in fact, an $\mathbb{F}_2$-vector space. 
\newline \underline {Case 2}. $A \neq \emptyset$ and $Z_\Gamma = \emptyset$.
\newline In this case necessarily $|A| \geq 2$ and $\bigcap_{a\in A}N^*(a)=Z_\Gamma = \emptyset$ is telling us that for any $v \in V(\Gamma)$ we have that $v$ is connected to a proper (possibly empty) subset of~$A$.
\newline \underline {Case 3}. $A \neq \emptyset$ and $Z_\Gamma \neq \emptyset$.
\newline In this case the condition $\bigcap_{a\in A}N^*(a)=Z_\Gamma$ is easy to understand.
\end{remark} 

\begin{notation}\label{notation_centra} For $F \subseteq W$, we denote by $C_W(F)$ the centralizer of $F$ in $W$. We allow the case $A = \emptyset$, in which case, by convention, we let $C_W(\{e_W\}) = W$, where $e_W$ is the identity of $W$. Also, we denote by $Z(W)$ the center of $W$.
\end{notation} 

    \begin{notation} For $w \in W$, we denote by $\mrm{ad}(w)$ the inner automorphism of $W$ induced by $w$.
    \end{notation}

We recall the following facts about right-angled Coxeter groups.
\begin{lemma}\label{finite_star_base_lemma}
Let $W=W(\Gamma)$ be a countable right-angled Coxeter group. 
\begin{enumerate}[(1)]
\item For every finite $A\subseteq V(\Gamma)$, $
C_W(A)=W_{\bigcap_{a\in A}N^*(a)}$ (notice that by \ref{the_convention} and \ref{notation_centra} if $A = \emptyset$ this equality becomes $C_W(A)=W_{V(\Gamma)} = W$ which is obviously true).
\item $ Z(W)=W_{Z_\Gamma}$ (recall \ref{Z_notation}).
\item For any family of subsets $\{A_i\}_{i \in I}$ of $V(\Gamma)$, we have $$ \bigcap_{i \in I}W_{A_i}= W_{\bigcap_{i \in I} A_i}$$
\end{enumerate}
\end{lemma}

\begin{proof}
For items (1) and (2), see e.g. \cite[Fact~3.52]{pao1}. For item (3), see \cite[Th.~4.1.6]{davis}.
\end{proof}

\begin{theorem}\label{theorem:inn_closed}
    Let $W=W(\Gamma)$ be a countable right-angled Coxeter group. Then the following are equivalent. \begin{enumerate}[(1)]
        \item $\mrm{Inn}(W(\Gamma))$ is a closed subgroup of $\mrm{Aut}(W(\Gamma))$.
        \item $\mrm{Inn}(W(\Gamma))$ is a discrete subgroup of $\mrm{Aut}(W(\Gamma))$.
        \item $\Gamma$ has the finite-star-base property.
    \end{enumerate}
\end{theorem}
\begin{proof}
    The group $\mrm{Inn}(W)$ is countable. Hence $\mrm{Inn}(W)$ is closed in $\mrm{Aut}(W)$ if and only if it is discrete. Indeed, $\mrm{Aut}(W)$ is Polish, and a countable closed subgroup $H$ of a Polish group $G$ is discrete by the Baire category theorem: being closed in $G$, the subgroup $H$ is itself a Polish space, hence a Baire space; since $H$ is countable, $H=\bigcup_{h\in H}\{h\}$ is a countable union of closed sets, so, as $H$ is not meager in itself, some singleton $\{h_0\}$ must have non-empty interior in $H$. Thus $h_0$ is isolated in $H$, and since left translation by $h_0^{-1}$ is a homeomorphism of $H$, the identity---and therefore, by translation, every point of $H$---is isolated; that is, $H$ is discrete. Conversely, a discrete subgroup of a metrizable topological group is closed. Therefore $\mrm{Inn}(W)$ is closed in $\mrm{Aut}(W)$ if and only if $\mrm{Inn}(W)$ is discrete if and only if every convergent sequence of inner automorphisms is eventually constant.

\smallskip \noindent    Now, suppose that there are vertices $x_1,\dots,x_n$ such that $\bigcap_{k=1}^n N^*(x_k) = Z_{\Gamma}$. If a sequence of inner automorphisms $\mrm{ad}(g_i)$ converges to an automorphism of $W(\Gamma)$, then for sufficiently large $i,j < \omega$ and for all $1 \leq k \leq n$, we have $g_i x_kg_i^{-1} = g_j x_k g_j^{-1}$, and thus $g_i g_j^{-1} \in C_W(x_k)$. But since $\bigcap_{k=1}^n N^*(x_k) = Z_{\Gamma}$, we have 
    $$\bigcap_{{k}=1}^n C_W(x_k) = \bigcap_{{k}=1}^n W_{N^*(x_k)} = W_{\bigcap_{{k}=1}^n   N^*(x_k)} = W_{Z_{\Gamma}} = Z(W).$$
    Therefore, for all sufficiently large $i,j < \omega$, we have $g_i g_j^{-1} \in Z(W)$ and thus $\mrm{ad}(g_i) = \mrm{ad}(g_j)$. 

    \smallskip \noindent Conversely, suppose that for all finite subsets $x_1,\dots,x_n$ of $V(\Gamma)$, the intersection $\bigcap_{k=1}^n N^*(x_k)$ properly contains $Z_{\Gamma}$. Enumerate all vertices of $\Gamma$ as $x_1,x_2,\dots$, and for each $n$, let $v_n$ be an element of $\bigcap_{k=1}^n N^*(x_k) \setminus Z_{\Gamma}$. Then the sequence $\mrm{ad}(v_n)$ converges to the identity, but no $\mrm{ad}(v_n)$ is the identity map since $v_n \not\in Z(W)$ for all $n < \omega$.
\end{proof}
\color{black}

\section{A characterization of $\mrm{Inn}(W)=\mrm{Spe}(W)$}\label{sec_Spe=Inn}

As noted in Section~\ref{sec_preliminaries}, for right-angled Coxeter groups of
arbitrary rank star-connectedness is necessary, but not sufficient, for
$\mrm{Inn}(W) = \mrm{Spe}(W)$. The aim of this section is to complete the
countable case by giving a graph-theoretic characterization of when this
equality holds. Theorem~\ref{theorem:inn_closed} already characterizes, for $W$
countable, when $\mrm{Inn}(W)$ is closed in $\mrm{Spe}(W)$. Here we prove that
star-connectedness implies that $\mrm{Inn}(W)$ is dense in $\mrm{Spe}(W)$, for
any right-angled Coxeter group $W$; together with
Theorem~\ref{theorem:inn_closed}, this settles the problem, yielding the desired
characterization of when $\mrm{Inn}(W) = \mrm{Spe}(W)$ for a countable
right-angled Coxeter group~$W$.

Much of the analysis in this section goes through at arbitrary
rank, so we drop our standing countability assumption: unless stated otherwise,
$W = W(\Gamma)$ is a right-angled Coxeter group on a graph $\Gamma$ of arbitrary
cardinality. The topology of pointwise convergence on $\mrm{Aut}(W)$, whose
basic neighbourhoods of the identity are the pointwise stabilizers of {\em
finite} subsets of $W$, still makes sense at this level of generality (see,
e.g., \cite[p.~136]{hodges}); it is, however, no longer separable once $\Gamma$
is uncountable, and we invoke countability only where the argument
requires it.

	\begin{context} Throughout this section $\Gamma$ is a graph of arbitrary rank (unless otherwise stated) and $W=W(\Gamma)$. For $w\in W$, we write $\mrm{sp}(w)$ for the set of vertices occurring in a reduced word for $w$; this is well-defined (cf. e.g. \cite[Section~3.3]{pao1}). We also write $\ell(w)$ for the word length of $w$ with respect to the standard generating set $V(\Gamma)$.
\end{context} 

\begin{definition}\label{locally_inner_def}
Let $\alpha\in\mrm{Spe}(W)$. We say that $\alpha$ is locally inner if for every finite $A\subseteq V(\Gamma)$ there is $w\in W$ such that $
\alpha(a)=waw^{-1}$
for every $a\in A$.
\end{definition}

Recall that, by \ref{Spe_closed}, $\mrm{Spe}(W)$ is closed in $\mrm{Aut}(W)$. Thus, in this section, when we write $\overline{\mrm{Inn}(W)}$ we mean the closure $\overline{\mrm{Inn}(W)}^{\,\mrm{Spe}(W)}$.

\begin{lemma}\label{locally_inner_closure_lemma}
For $\alpha\in\mrm{Spe}(W)$, the following are equivalent:
\begin{enumerate}[(1)]
\item $\alpha\in\overline{\mrm{Inn}(W)}$;
\item $\alpha$ is locally inner.
\end{enumerate}
Consequently,
\[
\overline{\mrm{Inn}(W)}=\mrm{Spe}(W)
\]
if and only if every element of $\mrm{Spe}(W)$ is locally inner.
\end{lemma}

\begin{proof}
The implication (1)$\Rightarrow$(2) is immediate from the pointwise convergence topology, by testing a basic neighbourhood of $\alpha$ on the finite set $A$.

\smallskip \noindent Conversely, suppose that $\alpha$ is locally inner. Let $F\subseteq W$ be finite. Choose a finite $A\subseteq V(\Gamma)$ containing every vertex appearing in some fixed word representative of an element of $F$. By local innerness, there is $w\in W$ such that $\alpha(a)=waw^{-1}$ for every $a\in A$. Since both maps are homomorphisms, it follows that $\alpha(f)=wfw^{-1}$ for every $f\in F$. Thus every basic neighbourhood of $\alpha$ meets $\mrm{Inn}(W)$, and hence $\alpha\in\overline{\mrm{Inn}(W)}$.
\end{proof}

The main result of this section is the following ``Density Theorem''.

\begin{theorem}\label{density_theorem}
Let $\Gamma$ be a star-connected graph and let $W=W(\Gamma)$. Then
\[
\overline{\mrm{Inn}(W)}=\mrm{Spe}(W).
\]
\end{theorem}

We prove Theorem~\ref{density_theorem} through a rank-two local conjugacy result (cf. \ref{rank_two_local_conjugacy_thm}) together with a Helly type result (cf. \ref{parabolic_coset_helly_lemma}). Notice that the proof of Theorem \ref{rank_two_local_conjugacy_thm} uses the finite-rank theorem of M{\"u}hlherr \cite{muhlherr} only in a {\em finite} induced subgraph~of~$\Gamma$.

\begin{fact}[{\cite[Lemma~2.17]{pao2}}]\label{support_propagation_fact}
Let $\alpha\in\mrm{Spe}(W)$, let $v\in V(\Gamma)$, and let $C$ be a connected component of $\Gamma\setminus N^*(v)$. If $x,y\in C$ and $v\in\mrm{sp}(\alpha(x))$, then
\[
v\in\mrm{sp}(\alpha(y)).
\]
\end{fact}

We also need the following elementary compatibility between special automorphisms and canonical retractions onto standard parabolic subgroups.

\begin{lemma}\label{special_retraction_lemma}
For $A\subseteq V(\Gamma)$, let
\[
\rho_A:W\to W_A
\]
be the canonical retraction of $W$ onto $W_A$, and put $M_A=\ker(\rho_A)$. Then
\[
M_A=\langle\!\langle V(\Gamma)\setminus A\rangle\!\rangle.
\]
Moreover, if $\alpha\in\mrm{Spe}(W)$, then
\[
\alpha(M_A)=M_A.
\]
Consequently $\alpha$ induces an automorphism
\[
\alpha_A\in\mrm{Aut}(W_A),
\]
given by
\[
\alpha_A(x)=\rho_A(\alpha(x))\qquad (x\in W_A),
\]
and in fact $\alpha_A\in\mrm{Spe}(W_A)$.
\end{lemma}

\begin{proof}
The equality $
M_A=\langle\!\langle V(\Gamma)\setminus A\rangle\!\rangle$
follows directly from the presentation of $W$. In more detail, quotienting $W$ by the normal closure of the generators outside $A$ kills precisely those generators and leaves the right-angled Coxeter presentation on the induced subgraph on $A$. The quotient map is exactly $\rho_A$.

\smallskip \noindent Now let $\alpha\in\mrm{Spe}(W)$. If $v\in V(\Gamma)\setminus A$, then $v\in M_A$. Since $\alpha \in \mrm{Spe}(W)$, there is $g\in W$ such that $
\alpha(v)=gvg^{-1}$.
Since $M_A$ is normal in  $W$, this gives $\alpha(v)\in M_A$. Therefore we have
\[
\alpha(V(\Gamma)\setminus A)\subseteq M_A,
\]
and hence
\[
\alpha(M_A)=\alpha\bigl(\langle\!\langle V(\Gamma)\setminus A\rangle\!\rangle\bigr)
=\langle\!\langle \alpha(V(\Gamma)\setminus A)\rangle\!\rangle\subseteq M_A.
\]
Applying the same argument to $\alpha^{-1}\in\mrm{Spe}(W)$ gives the reverse inclusion, so $\alpha(M_A)=M_A$.

\smallskip \noindent Thus $\alpha$ descends to an automorphism of $W/M_A$, which we identify with $W_A$ via the retraction $\rho_A$. This gives the formula
\[
\alpha_A(x)=\rho_A(\alpha(x))\qquad (x\in W_A).
\]
Finally, let $h\in W_A$ be an involution. Since $\alpha\in\mrm{Spe}(W)$, there is $g\in W$ such that
\[
\alpha(h)=ghg^{-1}.
\]
Applying $\rho_A$ gives
\[
\alpha_A(h)=\rho_A(\alpha(h))=\rho_A(g)h\rho_A(g)^{-1}.
\]
Since $\rho_A(g)\in W_A$, the involution $\alpha_A(h)$ is $W_A$-conjugate to $h$. Therefore $\alpha_A\in\mrm{Spe}(W_A)$.
\end{proof}

\begin{theorem}\label{rank_two_local_conjugacy_thm}
Assume that $\Gamma$ is star-connected. Let $\alpha\in\mrm{Spe}(W)$ and let $a,b\in V(\Gamma)$. Then there is $w\in W$ such that
\[
\alpha(a)=waw^{-1}
\qquad\text{and}\qquad
\alpha(b)=wbw^{-1}.
\]
\end{theorem}

\begin{proof}
If $a=b$, the assertion follows immediately from the definition of $\mrm{Spe}(W)$. Thus we may assume that $a\neq b$. Since $\alpha\in\mrm{Spe}(W)$, choose $u\in W$ such that
\[
\alpha(a)=uau^{-1}.
\]
Replace $\alpha$ by $\alpha_0=\mrm{ad}(u^{-1})\circ\alpha$. Then $\alpha_0\in\mrm{Spe}(W)$ and $\alpha_0(a)=a$. If the theorem is proved for $\alpha_0$, say
\[
\alpha_0(a)=t a t^{-1}
\qquad\text{and}\qquad
\alpha_0(b)=t b t^{-1},
\]
then $ut$ works for the original automorphism $\alpha$. Hence w.l.o.g. we assume that
\[
\alpha(a)=a.
\]

\smallskip \noindent Let
\[
S=\{a,b\}\cup\mrm{sp}(\alpha(b)).
\]
This is a finite subset of $V(\Gamma)$. Let $\Gamma_S$ be the induced subgraph on $S$, and let $W_S$ be the corresponding standard parabolic subgroup. By Lemma~\ref{special_retraction_lemma}, the retraction $\rho_S:W\to W_S$ gives an induced special automorphism $\beta_S\in\mrm{Spe}(W_S)$ such that
\[
\beta_S(x)=\rho_S(\alpha(x))\quad (x\in W_S).
\]
Moreover,
\[
\beta_S(a)=a
\qquad\text{and}\qquad
\beta_S(b)=\alpha(b),
\]
because $\alpha(a)=a$ and $\mrm{sp}(\alpha(b))\subseteq S$.

\smallskip \noindent We first prove the key graph-theoretic claim.

\smallskip \noindent
\begin{claim}\label{crucial_claim1} For every $s\in S$, $
\{a,b\}\cap N^*_{\Gamma_S}(s)\neq\emptyset$.
\end{claim}

\medskip \noindent
\begin{claimproof} Suppose not. In particular, $s\neq a,b$. Since $s\in S\setminus\{a,b\}$, it follows that
\[
s\in\mrm{sp}(\alpha(b)).
\]
By the assumption, both $a$ and $b$ belong to $\Gamma\setminus N_{\Gamma}^*(s)$. Since $\Gamma$ is star-connected, they lie in the same connected component of $\Gamma\setminus N_{\Gamma}^*(s)$. By Fact~\ref{support_propagation_fact}, applied to this component and to the fact that $s\in\mrm{sp}(\alpha(b))$, we obtain
\[
s\in\mrm{sp}(\alpha(a)).
\]
But $\alpha(a)=a$, so $\mrm{sp}(\alpha(a))=\{a\}$. Hence $s=a$, contradicting $a \notin N^*_\Gamma(s)$. This proves the claim.
\end{claimproof}

\medskip \noindent Now we use M{\"u}hlherr's finite-rank theorem (i.e., the main result of \cite{muhlherr}) in the finite graph $\Gamma_S$. Since $S$ is finite, Fact~\ref{partial_conj_fact} says that $\mrm{Spe}(W_S)$ is generated by the partial conjugations of $W_S$. Let $\pi(s,C)$ be a partial conjugation of $W_S$, where $s\in S$ and $C$ is a union of connected components of 
$\Gamma_S\setminus N^*_{\Gamma_S}(s)$.
By Claim~\ref{crucial_claim1}, at least one of $a,b$ belongs to $N^*_{\Gamma_S}(s)$. 

\smallskip

\begin{claim}\label{crucial_claim2} If $s \in S$, then $\pi(s,C)$ is inner on the pair $\{a,b\}$.
\end{claim}

\smallskip

\begin{claimproof} If neither $a$ nor $b$ belongs to $C$, then $\pi(s,C)$ fixes both $a$ and $b$, so it agrees with the identity on $\{a,b\}$. If one of $a,b$ belongs to $C$, then the other belongs to $N^*_{\Gamma_S}(s)$, because at least one of $a,b$ belongs to $N^*_{\Gamma_S}(s)$ and $C$ is disjoint from $N^*_{\Gamma_S}(s)$. In this case $\pi(s,C)$ agrees with conjugation by $s$ on $\{a,b\}$: it sends the element in $C$ to its $s$-conjugate, and it fixes the element in $N^*_{\Gamma_S}(s)$, as does conjugation by $s$.
\end{claimproof} 

\smallskip
\noindent Let now $\mrm{Spe}_{\{a,b\}}(W_S)$ be the following set:
\[
\{\varphi\in\mrm{Spe}(W_S):\text{ there is }x \in W_S\text{ such that }\varphi(z)=xzx^{-1}\text{ for }z=a,b\}.
\]
\smallskip
\noindent
\begin{claim}\label{crucial_claim3} $\mrm{Spe}_{\{a,b\}}(W_S)$ is a subgroup of $\mrm{Spe}(W_S)$. 
\end{claim}

\smallskip
\noindent
\begin{claimproof}  If $\varphi$ is inner on $\{a,b\}$ via $x$ and $\psi$ is inner on $\{a,b\}$ via $y$, then for $z\in\{a,b\}$ we have
\[
(\varphi\circ\psi)(z)=\varphi(yzy^{-1})=\varphi(y)xzx^{-1}\varphi(y)^{-1},
\]
so $\varphi\circ\psi$ is inner on $\{a,b\}$ via $\varphi(y)x$. Also, if $\varphi$ is inner on $\{a,b\}$ via $x$, then
\[
z=\varphi^{-1}(x)\varphi^{-1}(z)\varphi^{-1}(x)^{-1}
\]
for $z\in\{a,b\}$, and so $\varphi^{-1}$ is inner on $\{a,b\}$ via $\varphi^{-1}(x)^{-1}$.
\end{claimproof}

\smallskip \noindent Now, by Claim~\ref{crucial_claim2} every partial conjugation of $W_S$ belongs to $\mrm{Spe}_{\{a,b\}}(W_S)$. Since $\mrm{Spe}(W_S)$ is generated by these partial conjugations (cf. Fact~\ref{partial_conj_fact}) and $\mrm{Spe}_{\{a,b\}}(W_S)$ is a subgroup of $\mrm{Spe}(W_S)$, we have $\mrm{Spe}_{\{a,b\}}(W_S) = \mrm{Spe}(W_S)$. In particular, $\beta_S \in \mrm{Spe}_{\{a,b\}}(W_S)$. Therefore, there is $w\in W_S$ such that
\[
\beta_S(a)=waw^{-1}
\qquad\text{and}\qquad
\beta_S(b)=wbw^{-1}.
\]
Using $\beta_S(a) = a$ and $\beta_S(b) = \alpha(b)$, we get
\[
a=waw^{-1}
\qquad\text{and}\qquad
\alpha(b)=wbw^{-1}.
\]
Recalling that w.l.o.g. we were assuming that $\alpha(a) = a$, we are done.
\end{proof}

We now pass from two vertices to finitely many vertices using the following Helly property for cosets of standard parabolic subgroups, i.e., Lemma \ref{parabolic_coset_helly_lemma}.

\begin{lemma}\label{lemma_special_subgroup_convex}
    For any $S \subseteq V(\Gamma)$, the Cayley graph $\mrm{Cay}(W_{S},S)$ is a convex subgraph of $\mrm{Cay}(W(\Gamma),V(\Gamma))$.
\end{lemma}
\begin{proof}
    This follows from the fact that any shortest word representing an element of the standard parabolic subgroup $W_{S}$ contains only elements of $S$ (see \cite[Corollary 4.1.2]{davis}).
\end{proof}

\begin{definition}
    Let $(X,d)$ be a metric space. For any two points $x,y \in X$, we let $$I(x,y) = \{z \in X : d(x,z) + d(z,y) = d(x,y) \}.$$ $X$ is called \textit{median} if for all $x,y,z \in X$, there is a unique point in the intersection $$I(x,y) \cap I(x,z) \cap I(y,z).$$
\end{definition}

\begin{lemma}\label{lemma_median_cayley_graph}
    The Cayley graph $\mrm{Cay}(W(\Gamma),V(\Gamma))$ is a median space. 
\end{lemma}
\begin{proof}
    If $V(\Gamma)$ is finite, the Cayley graph $\mrm{Cay}(W(\Gamma),V(\Gamma))$ is the $1$-skeleton of the Davis complex of $W$ (see \cite[Proposition 7.3.4]{davis}). As $W$ is right-angled, its Davis complex is a $\mrm{CAT}(0)$ cube complex (see \cite[Theorem 12.2.1]{davis}), and hence its $1$-skeleton is a median graph.

    \smallskip \noindent It follows from the finite-rank case that $\mrm{Cay}(W(\Gamma),V(\Gamma))$ is median, even when $V(\Gamma)$ is infinite. Indeed, for any three vertices $x,y,z \in W(\Gamma)$, let $S = \mrm{sp}(x) \cup \mrm{sp}(y) \cup \mrm{sp}(z)$. Since standard parabolic subgroups are convex (Lemma \ref{lemma_special_subgroup_convex}), the intervals $I(x,y)$, $I(x,z)$, and $I(y,z)$ all lie in the subgraph $\mrm{Cay}(W_{S},S)$. Since $S$ is finite, these intervals have a unique intersection by the median property in the finite-rank case.
\end{proof}

\begin{lemma}\label{parabolic_coset_helly_lemma}
Let $X_1,\ldots,X_n\subseteq V(\Gamma)$, and let
\[
g_iW_{X_i}\qquad (1\leq i\leq n)
\]
be left cosets of standard parabolic subgroups of $W$. If these cosets intersect pairwise, then
\[
\bigcap_{i=1}^n g_iW_{X_i}\neq\emptyset.
\]
\end{lemma}

\begin{proof}
By Lemma \ref{lemma_median_cayley_graph}, the Cayley graph $\mrm{Cay}(W(\Gamma), V(\Gamma))$ is a median space. By Lemma \ref{lemma_special_subgroup_convex}, each $\mrm{Cay}(W_{X_i}, X_i)$ is a convex subgraph, and hence so are each of the translates $g_i\mrm{Cay}(W_{X_i}, X_i)$. Convex subsets of median graphs are gated, and so enjoy the finite Helly property: any finite family of pairwise intersecting gated sets has a non-empty intersection (cf. \cite{helly} and references therein). This intersection must contain a vertex of $\mrm{Cay}(W(\Gamma), V(\Gamma))$, i.e., an element of $W(\Gamma)$. Therefore, the pairwise intersecting cosets $g_1W_{X_1},\dots,g_nW_{X_n}$ from the statement of the lemma have a common element.

\end{proof}

\begin{proof}[Proof of Theorem~\ref{density_theorem}]
First suppose that $\Gamma$ is star-connected. We prove that every element of $\mrm{Spe}(W)$ is locally inner. Let $\alpha\in\mrm{Spe}(W)$ and let $A\subseteq V(\Gamma)$ be finite. For each $a\in A$, set
\[
K_a=\{w\in W:\alpha(a)=waw^{-1}\}.
\]
Since $\alpha$ is special, each $K_a$ is non-empty. Fix $g_a\in K_a$. For $w\in W$ we have $\alpha(a)=waw^{-1}$ if and only if $waw^{-1}=g_aag_a^{-1}$, i.e.\ $(g_a^{-1}w)\,a\,(g_a^{-1}w)^{-1}=a$, i.e.\ $g_a^{-1}w\in C_W(a)$; equivalently $w\in g_aC_W(a)$. Hence, using $C_W(a)=W_{N^*(a)}$ (use Lemma~\ref{finite_star_base_lemma}(1) for $A=\{a\}$), we have that
\[
K_a=g_aC_W(a)=g_aW_{N^*(a)},
\]
so $K_a$ is a left coset of a standard parabolic subgroup.

\smallskip \noindent By Theorem~\ref{rank_two_local_conjugacy_thm}, for every $a,b\in A$ there is $w\in W$ such that
\[
\alpha(a)=waw^{-1}
\qquad\text{and}\qquad
\alpha(b)=wbw^{-1}.
\]
Equivalently, $K_a\cap K_b\neq\emptyset$. Thus the finite family $(K_a:a\in A)$ is pairwise intersecting, and so, by Lemma~\ref{parabolic_coset_helly_lemma}, we conclude that $\bigcap_{a\in A}K_a\neq\emptyset$.
Choose $w$ in this intersection. Then $\alpha(a)=waw^{-1}$ for every $a\in A$. Hence $\alpha$ is locally inner. By Lemma~\ref{locally_inner_closure_lemma}, $\alpha\in\overline{\mrm{Inn}(W)}$. Since $\alpha\in\mrm{Spe}(W)$ was arbitrary, we get the desired
\[
\overline{\mrm{Inn}(W)}=\mrm{Spe}(W).
\]
\end{proof}

We can now combine the ``Density Theorem'' (Theorem~\ref{density_theorem}) with Theorem~\ref{theorem:inn_closed}, to prove Theorem \ref{theorem_spe_inn}.

\begin{proof}[Proof of Theorem \ref{theorem_spe_inn}]
Assume first that $\mrm{Inn}(W)=\mrm{Spe}(W)$. Then $\Gamma$ is star-connected, since this is a necessary condition for $\mrm{Inn}(W)=\mrm{Spe}(W)$ (cf. e.g. \cite[Theorem~2.8]{pao2}). Furthermore, Proposition \ref{Spe_closed} implies $\mrm{Spe}(W)=\mrm{Inn}(W)$ is closed, so Theorem~\ref{theorem:inn_closed} implies $\Gamma$ has the finite-star-base property. Thus \eqref{theorem_spe_inn_1} implies \eqref{theorem_spe_inn_2}.

\smallskip \noindent Conversely, assume \eqref{theorem_spe_inn_2}. Since $\Gamma$ is star-connected, Theorem~\ref{density_theorem} gives
\[
\overline{\mrm{Inn}(W)}=\mrm{Spe}(W).
\]
Since $\Gamma$ has the finite-star-base property, Theorem~\ref{theorem:inn_closed} implies
that $\mrm{Inn}(W)$ is closed in $\mrm{Spe}(W)$. Thus,
\[
\mrm{Spe}(W)=\overline{\mrm{Inn}(W)}=\mrm{Inn}(W).
\]
Therefore \eqref{theorem_spe_inn_2} implies \eqref{theorem_spe_inn_1}.

\smallskip \noindent Conditions \eqref{theorem_spe_inn_2} and \eqref{theorem_spe_inn_3} are equivalent by Theorem \ref{theorem:inn_closed}.
\end{proof}

\section{Proof of the main theorem}\label{section_main_theorem}

The purpose of this section is to prove Theorem \ref{thm:main} and Theorem \ref{thm:stronger}. The idea of the proof is to construct, for every non-Archimedean Polish group $G$, a countable graph $\hat{\Gamma}$ such that the  Tits decomposition yields $$\mrm{Aut}(W(\hat{\Gamma})) = \mrm{Inn}(W(\hat{\Gamma})) \rtimes G.$$ 

We now recall the definition of a projective plane, since this will be relevant for the construction of the graph $\hat{\Gamma}$. We will only consider non-degenerate projective planes.

\begin{notation}
    A (non-degenerate) projective plane $P$ is a pair of sets $A,B$ (called \textit{points} and \textit{lines}, respectively) together with a relation on the set $A \cup B$ called \textit{incidence} satisfying the following: \begin{enumerate}[(1)]
        \item Given any two distinct points $p_1,p_2\in A$, there is a unique line $\ell$ incident to both of them, denoted $p_1\vee  p_2$.
        \item Given any two distinct lines $\ell_1, \ell_2 \in B$, there is a unique line incident to both of them, denoted $\ell_1 \wedge \ell_2$.
        \item There are at least four points, with no three belonging to the same line.
    \end{enumerate}
    For any projective plane $P$, we denote by $\Gamma(P)$ the associated (bipartite) graph, with vertex set $A \cup B$, and an edge between $p \in A$ and $\ell \in B$ if and only if $p \in \ell$.
\end{notation}

\begin{lemma}\label{lemma_proj_plane_graph_properties}
    If $P$ is a (non-degenerate) projective plane, then $\Gamma(P)$ has the following properties:
    \begin{enumerate}[(1)]
        \item $\Gamma(P)$ is bipartite; equivalently, $\Gamma(P)$ contains no odd-cycles;
        \item $\Gamma(P)$ is square-free;
        \item $\Gamma(P)$ is star-connected;
        \item $\Gamma(P)$ has the star-property;
        \item $\Gamma(P)$ contains the path $P_4$ as an induced subgraph.
    \end{enumerate}
\end{lemma}
\begin{proof}
    We only verify that $\Gamma(P)$ is star-connected. By duality it suffices to deal only with the case of points. Let $p$ be a point and let $\Gamma \setminus N^*(p) = Q$. Let $x$ and $y$ be points of $Q$. If $\ell_1 = x \vee y$ does not contain $p$ we are done, since then $x \rightarrow \ell_1 \rightarrow y$ is a path in $Q$. Suppose then that $\ell_1 \notin Q$, i.e., that $\ell_1$ contains the point $p$. Since $P$ is non-degenerate we can find $z \notin \ell_1$. Notice that $z \notin \{x, y\}$. Let $\ell_2 = x \vee z$, then $\ell_2$ does not contain $p$, since otherwise $z \in \ell_1$. Let $\ell_3 = y \vee z$, then $\ell_3$ does not contain $p$ (argue as before) and $\ell_3 \neq \ell_2$, since otherwise $\ell_3 = \ell_2 = x \vee y = \ell_1$. This shows that $x \rightarrow x \vee z \rightarrow z \rightarrow y \vee z \rightarrow y$ is a path in $Q$. Let $\ell_1 \neq \ell_2 \in Q$ be lines, then $\ell_1 \wedge \ell_2 \neq p$ and so $\ell_1 \rightarrow \ell_1 \wedge \ell_2 \rightarrow \ell_2$ is a path in $Q$.
\end{proof}

\begin{lemma}\label{lemma_graph_construction}
    For every graph $\Gamma$ there exists a graph $\hat{\Gamma}$ of size $|\Gamma| + \aleph_0$ such that the following holds:
	\begin{enumerate}[(1)]
	\item $\hat{\Gamma}$ is star-connected, it is triangle-free, it contains the path $P_4$ as an induced subgraph and it has the star-property;
	\item $\mrm{Aut}(\Gamma) \cong_{\mrm{top}} \mrm{Aut}(\hat{\Gamma})$.
    \item The map $\Gamma \mapsto \hat{\Gamma}$ is Borel.
    \end{enumerate}

\end{lemma}
\begin{proof}
    Let $P=P_{\Gamma}^*$ be the projective plane constructed in \cite[Theorem 3]{pao3}. Although it is not explicitly stated in \cite{pao3}, it is clear from the construction that the isomorphism $\mrm{Aut}(\Gamma) \cong \mrm{Aut}(P)$ exhibited there is topological. Let $\Gamma'=\Gamma(P)$ be the associated bipartite graph, let $A \subseteq V(\Gamma')$ be the set of vertices corresponding to points of $P$, and let $B \subseteq V(\Gamma')$ be the set of vertices corresponding to lines of $P$. Clearly $\mrm{Aut}(P) \leq \mrm{Aut}(\Gamma')$, but a priori there may be automorphisms of $\Gamma'$ that interchange $A$ and $B$. To remedy this, we create a new graph $\hat{\Gamma}$ by replacing each $p \in A, \ell \in B$ which satisfy $p R_{\Gamma'}\ell$ with the graph $L_{p,\ell}$ in Figure \ref{fig_graph}.

    \begin{figure}[ht]
    \centering
    \begin{tikzpicture}[
    every node/.style={circle,fill=black,inner sep=1.5pt},
    scale=1
]
\usetikzlibrary{calc}

\node (t0) at (0,1) {};
\node (b0) at (0,0) {};
\node (t1) at (1,1) {};
\node (b1) at (1,0) {};

\node (m2) at (2,1) {};
\node (t3) at (3,1) {};
\node (b3) at (3,0) {};

\node (m4) at (4,1) {};
\node (t5) at (5,1) {};
\node (b5) at (5,0) {};

\node (m6) at (6,1) {};
\node (t7) at (7,1) {};
\node (b7) at (7,0) {};

\draw (t0)--(t1)--(m2)--(t3)--(m4)--(t5)--(m6)--(t7)--++(0.45,0);

\draw (b0)--(b1)--(b3)--(b5)--(b7)--++(0.45,0);

\draw (t0)--(b0);
\draw (t1)--(b1);
\draw (t3)--(b3);
\draw (t5)--(b5);
\draw (t7)--(b7);

\node[draw=none,fill=none,left=6pt] at (t0) {$p$};
\node[draw=none,fill=none,left=6pt] at (b0) {$\ell$};

\node[draw=none,fill=none,right=7pt] at ($(t7)+(0.45,0)$) {$\cdots$};
\node[draw=none,fill=none,right=7pt] at ($(b7)+(0.45,0)$) {$\cdots$};

\end{tikzpicture}
    \caption{The subgraph $L_{p,\ell}$ of $\hat{\Gamma}$.}
    \label{fig_graph}
    \end{figure}
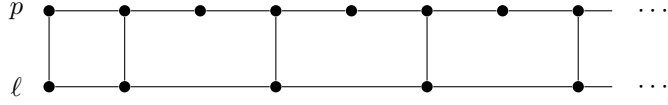

    We claim $\mrm{Aut}(\hat{\Gamma}) \cong_{\mrm{top}} \mrm{Aut}(P)$. 
    There is an obvious continuous map $$\Psi:\mrm{Aut}(P) \to \mrm{Aut}(\hat{\Gamma}).$$ To show $\Psi$ is a homeomorphism, we construct the inverse $\Theta : \mrm{Aut}(\hat{\Gamma}) \to \mrm{Aut}(P)$.
    Since $\Gamma'$ is square free (Lemma \ref{lemma_proj_plane_graph_properties}), any  $\sigma \in \mrm{Aut}(\hat{\Gamma})$ must take the square in $L_{p,\ell}$ to the square in $L_{p',\ell'}$ for some $p' \in A, \ell' \in B$. Since $\Gamma'$ is pentagon free (again Lemma \ref{lemma_proj_plane_graph_properties}), it follows that $\sigma(L_{p,\ell}) = L_{p',\ell'}$, and moreover that $\sigma(p) = p'$ and $\sigma(\ell) = \ell'$. We therefore define $\Theta(\sigma) \coloneqq  \sigma\vert_{\Gamma'}$. It is not hard to see that $\Theta$ is continuous and that $\Theta$ and $\Psi$ are inverses. This proves the claim.

    \smallskip \noindent We conclude $\mrm{Aut}(\hat{\Gamma}) \cong_{\mrm{top}} \mrm{Aut}(P) \cong_{\mrm{top}} \mrm{Aut}(\Gamma)$. Since $\Gamma'$ is star-connected, triangle free, contains $P_4$ as an induced subgraph, and has the star-property, it is easy to see the same is true of $\hat{\Gamma}$. 
    
    \smallskip \noindent Finally, since the assignment $\Gamma \mapsto P_{\Gamma}^*$ is Borel (see \cite[Theorem 3]{pao3}), the assignment $P \mapsto \Gamma'$ is Borel, and the assignment $\Gamma'\mapsto \hat{\Gamma}$ can be done in a Borel way, it follows that the composition $\Gamma \mapsto \hat{\Gamma}$ is Borel. 
\end{proof}

	\begin{fact}\label{second_fact} For every non-Archimedean Polish group $G$ there is a countable graph $\Gamma$ such that $G \cong_{\mrm{top}} \mrm{Aut}(\Gamma)$. 
\end{fact}

	\begin{proof} This is probably folklore but for an explicit construction cf. \cite[Theorem~13.1.2]{GaoBook}, recalling that any non-Archimedean Polish group is the group of automorphisms of a countable relational structure. Notice that the preservation of the automorphism group is not explicitly stated in \cite[Theorem~13.1.2]{GaoBook} but it is immediate to see that this indeed holds, using that the $n$-tags used in the proof of \cite[Theorem~13.1.2]{GaoBook} are rigid objects together with the other properties of the construction. Notice that, to be precise, in \cite[Theorem~13.1.2]{GaoBook} it is assumed that the relational language $L = \{R_n : 0 < n < \omega\}$ has the property that $R_n$ has arity $n$, but by using the trick employed in the proof of \cite[Theorem~3.6.1]{GaoBook} this can be assumed w.l.o.g.; notice in fact that the structure constructed in the proof of \cite[Theorem~3.6.1]{GaoBook} (i.e., the one in the language $L$, in the notation of the proof) is bidefinable with the original one (i.e., the one in the language $L_0$).
	\end{proof}
	We denote by $\mathbf{Cl}(\mrm{Sym}(\omega))$ the Borel space of closed subgroups of $\mrm{Sym}(\omega)$ (cf. e.g.~\cite{nies}) and by $\mrm{Groups}_\omega$ the Borel space of groups with domain $\omega$ (cf. e.g.~\cite[Chapter~11]{GaoBook}).

\begin{theorem}\label{thm:stronger}
    There is a Borel map $f : \mathbf{Cl}(\mrm{Sym}(\omega)) \to \mrm{Groups}_\omega$ such that for all $G \in \mathbf{Cl}(\mrm{Sym}(\omega))$, the group $f(G)$ satisfies:
    \begin{enumerate}[(1)]
        \item\label{part:embed} There is a closed subgroup $\mathcal{G} \leq \mrm{Aut}(f(G))$ topologically isomorphic to $G$.
        \item \label{part:inn_closed} $\mrm{Inn}(f(G))$ is a closed (equivalently, discrete) subgroup of $\mrm{Aut}(f(G))$.
        \item \label{part:split} There is a topological decomposition $\mrm{Aut}(f(G))= \mrm{Inn}(f(G)) \rtimes \mathcal{G}$.
    \end{enumerate}
\end{theorem}
\begin{proof}
    For $G \in \mathbf{Cl}(\mrm{Sym}(\omega))$, let $\Gamma_G$ be the graph constructed in Fact \ref{second_fact}, let $\hat{\Gamma_G}$ be the graph constructed in Lemma \ref{lemma_graph_construction}, and define $f(G) = W(\hat{\Gamma_G})$.
    Since $\hat{\Gamma_G}$ has the star-property by Lemma \ref{lemma_graph_construction}, Fact \ref{fact_F_gamma_star_property} implies $F(\hat{\Gamma_G}) = \mrm{Aut}(\hat{\Gamma_G}) \cong_{\mrm{top}} G$. We let $\mathcal{G} = F(\hat{\Gamma_G})$. This establishes claim \eqref{part:embed}.
    By Lemma \ref{lemma_graph_construction}, $\hat{\Gamma_G}$ is star-connected, triangle-free, and contains $P_4$ as a subgraph. Thus, \cite[Theorem~2.19]{pao2} implies that $\mrm{Spe}(W(\hat{\Gamma_G})) = \mrm{Inn}(W(\hat{\Gamma_G}))$. This establishes claim \eqref{part:inn_closed}.

    \smallskip \noindent Theorem \ref{theorem_tits_topological} then implies that the Tits decomposition $$\mrm{Aut}(W(\hat{\Gamma_G}))= \mrm{Spe}(W(\hat{\Gamma_G}))\rtimes F(\hat{\Gamma_G}) =
    \mrm{Inn}(W(\hat{\Gamma_G})) \rtimes \mathcal{G} $$
    is topological, establishing claim \eqref{part:split}.  

   \smallskip \noindent Finally, concerning Borelness of the construction, notice the following:
\begin{enumerate}[(1)]
	\item passing from a closed subgroup of $\mrm{Sym}(\omega)$ to its canonical structure (cf. e.g. \cite[Proof of Theorem~4.1.4]{hodges}) is well-known to be Borel;
	\item the construction from Fact~\ref{second_fact} is Borel (cf. \cite[Theorem~13.1.2]{GaoBook});
	\item the construction from Lemma \ref{lemma_graph_construction} is Borel;
	\item the construction $\Gamma \mapsto W(\Gamma)$ is easily seen to be Borel.
\end{enumerate}

\smallskip \noindent Composing (1)--(4), we conclude that the map $G \mapsto W(\hat{\Gamma_G})$ is Borel, as desired.

\end{proof}

    Notice also that, by Proposition~\ref{Spe_closed}, in our construction
$\mrm{Inn}(W(\hat{\Gamma_G}))$ is closed in $\mrm{Aut}(W(\hat{\Gamma_G}))$,
and hence $\mrm{Inn}(W(\hat{\Gamma_G}))$ is discrete by Theorem~\ref{theorem:inn_closed}.

\begin{proof}[Proof of Theorem \ref{thm:main}]
    Suppose $G$ is a Hausdorff topological group which realizes as $\mrm{Out}(W) = \mrm{Aut}(W)/\mrm{Inn}(W)$ in the quotient topology for a countable group $W$. Since $G$ is Hausdorff, $\mrm{Inn}(W)$ must be closed in $\mrm{Aut}(W)$. Since $W$ is countable, $\mrm{Aut}(W)$ is a non-Archimedean Polish group. It is well known that the quotient of a Polish group by a closed normal subgroup is Polish (see, e.g., \cite[Proposition 1.2.3]{Becker_Kechris_1996}). Further, the image of a countable neighborhood basis of open subgroups of $\mrm{Aut}(W)$ is a countable neighborhood basis of open subgroups of $\mrm{Out}(W)$. Thus, $G\cong_{\mrm{top}} \mrm{Out}(W)$ is a non-Archimedean Polish group.
    
    \smallskip \noindent Conversely, let $G$ be a non-Archimedean Polish group. We may assume $G \in \mathbf{Cl}(\mrm{Sym}(\omega))$, and we will show $\mrm{Out}(f(G)) \cong_{\mrm{top}} G$. By statement \eqref{part:embed} of Theorem \ref{thm:stronger}, it is enough to show $\mrm{Out}(f(G)) \cong_{\mrm{top}} \mathcal{G}$. By statement \eqref{part:split} of Theorem \ref{thm:stronger}, there is a projection homomorphism \mbox{$\pi: \mrm{Aut}(f(G)) \to \mathcal{G}$} such that $\pi$ is a continuous open map. Therefore, the induced continuous homomorphism  \mbox{$\widetilde{\pi} : \mrm{Out}(f(G)) \to \mathcal{G}$} is open, and hence it is a topological isomorphism.
\end{proof}

\begin{remark}
    Note that the proofs of Theorem \ref{thm:main} and Theorem \ref{thm:stronger} do not rely on our new results Theorem \ref{theorem_spe_inn} and Theorem \ref{density_theorem}, only on the sufficient conditions for $\mrm{Spe}(W) = \mrm{Inn}(W)$ introduced in \cite[Theorem~2.19]{pao2}. While these new results could simplify the proofs of Theorems \ref{thm:main}  and \ref{thm:stronger} (in particular the construction of $\hat{\Gamma}$ in Lemma \ref{lemma_graph_construction}), we have chosen to use only \cite[Theorem~2.19]{pao2} to emphasize the independence of the results.
\end{remark}

\bibliographystyle{amsplain}
\bibliography{references.bib}

@article{castella,
  author  = {A. Castella},
  title   = {Sur les Automorphismes et la Rigidit{\'e} des Groupes de {C}oxeter {\`a} Angles Droits},
  journal = {J. Algebra},
  volume  = {301},
  number  = {2},
  pages   = {642--669},
  year    = {2006},
}

@book{hodges,
  author    = {W. Hodges},
  title     = {{Model Theory}},
  publisher = {Cambridge University Press},
  address   = {Cambridge},
  year      = {1993},
}

@article{pao1,
  author  = {B. M{\"u}hlherr and G. Paolini and S. Shelah},
  title   = {First-order aspects of {C}oxeter groups},
  journal = {J. Algebra},
  volume  = {595},
  pages   = {297--346},
  year    = {2022},
}

@article{franzsen-howlett-muhlherr,
  author  = {W. N. Franzsen and R. B. Howlett and B. M{\"u}hlherr},
  title   = {Reflections in Abstract {C}oxeter Groups},
  journal = {Comment. Math. Helv.},
  volume  = {81},
  number  = {3},
  pages   = {665--697},
  year    = {2006},
}

@article{gao-nies-paolini,
  author  = {S. Gao and A. Nies and G. Paolini},
  title   = {Procountable Groups are not Classifiable by Countable Structures},
  journal = {Preprint},
  note    = {\url{https://arxiv.org/abs/2512.12256}},
  year    = {2025},
}

@book{GaoBook,
  author    = {Su Gao},
  title     = {{Invariant Descriptive Set Theory}},
  publisher = {CRC Press},
  year      = {2009},
}

@article{pao2,
  author  = {T. Hyttinen and G. Paolini},
  title   = {{C}oxeter groups and abstract elementary classes: the right-angled case},
  journal = {Notre Dame J. Form. Log.},
  volume  = {60},
  number  = {4},
  pages   = {707--731},
  year    = {2019},
}

@article{nies,
  author  = {A. Kechris and A. Nies and K. Tent},
  title   = {The complexity of topological group isomorphism},
  journal = {J. Symb. Log.},
  volume  = {83},
  number  = {3},
  pages   = {1190--1203},
  year    = {2018},
}

@article{muhlherr,
  author  = {B. M{\"u}hlherr},
  title   = {Automorphisms of Graph-Universal {C}oxeter Groups},
  journal = {J. Algebra},
  volume  = {200},
  pages   = {629--649},
  year    = {1998},
}

@article{pao3,
  author  = {G. Paolini},
  title   = {The class of non-{D}esarguesian projective planes is {B}orel complete},
  journal = {Proc. Am. Math. Soc.},
  volume  = {146},
  pages   = {4927--4936},
  year    = {2018},
}

@article{tits,
  author  = {Jacques Tits},
  title   = {Sur le Groupe des Automorphismes de Certains Groupes de {C}oxeter},
  journal = {J. Algebra},
  volume  = {113},
  number  = {2},
  pages   = {346--357},
  year    = {1988},
}

@book{davis,
  author    = {Michael W. Davis},
  title     = {{The Geometry and Topology of Coxeter Groups}},
  series    = {London Mathematical Society Monographs Series},
  volume    = {32},
  publisher = {Princeton University Press},
  address   = {Princeton, NJ},
  year      = {2008},
}

@article{helly,
  author  = {J. Chalopin and V. Chepoi and A. Genevois and H. Hirai and D. Osajda},
  title   = {{H}elly groups},
  journal = {Geom. Topol.},
  volume  = {29},
  number  = {1},
  pages   = {1--70},
  year    = {2025},
}

@article{Matumoto1989,
  author  = {Takao Matumoto},
  title   = {Any group is represented by an outer automorphism group},
  journal = {Hiroshima Math. J.},
  volume  = {19},
  year    = {1989},
  pages   = {209--219},
  doi     = {10.32917/hmj/1206129490}
}

@book{Becker_Kechris_1996, 
place={Cambridge}, 
series={London Mathematical Society Lecture Note Series}, 
title={{The Descriptive Set Theory of Polish Group Actions}}, 
publisher={Cambridge University Press}, 
author={Becker, Howard and Kechris, Alexander S.}, 
year={1996}, 
collection={London Mathematical Society Lecture Note Series}}

@article{bumagin_wise,
title = "Every group is an outer automorphism group of a finitely generated group",
author = "Inna Bumagin and Wise, Daniel T.",
year = "2005",
month = aug,
day = "1",
doi = "10.1016/j.jpaa.2004.12.033",
volume = "200",
pages = "137--147",
journal = "J. Pure Appl. Algebra",
issn = "0022-4049",
publisher = "Elsevier B.V.",
number = "1-2",
}

@book{DixonMortimer1996,
  author    = {John D. Dixon and Brian Mortimer},
  title     = {{Permutation Groups}},
  series    = {Graduate Texts in Mathematics},
  volume    = {163},
  publisher = {Springer-Verlag},
  address   = {New York},
  year      = {1996},
  isbn      = {978-0-387-94599-9}
}

@article{SILVER19801,
title = {Counting the number of equivalence classes of {B}orel and coanalytic equivalence relations},
journal = {Ann. Math. Logic},
volume = {18},
number = {1},
pages = {1-28},
year = {1980},
issn = {0003-4843},
doi = {https://doi.org/10.1016/0003-4843(80)90002-9},
url = {https://www.sciencedirect.com/science/article/pii/0003484380900029},
author = {Jack H. Silver}
}

@article{DROSTE_GIRAUDET_GOBEL_2001, 
title={ALL GROUPS ARE OUTER AUTOMORPHISM GROUPS OF SIMPLE GROUPS}, volume={64}, DOI={10.1112/S0024610701002484}, number={3}, journal={J. Lond. Math. Soc.}, author={Droste, Manfred and Giraudet, Mich{\`e}le and G{\"o}bel, R{\"u}diger}, year={2001}, pages={565–575}}

@article{GoebelParas2000,
  title   = {Outer automorphism groups of metabelian groups},
  author  = {G{\"o}bel, R{\"u}diger and Paras, Agnes T.},
  journal = {J. Pure Appl. Algebra},
  volume  = {149},
  number  = {3},
  pages   = {251--266},
  year    = {2000},
  publisher = {Elsevier},
  doi     = {10.1016/S0022-4049(99)00022-5}
}
\end{document}